\documentclass[12pt,reqno]{amsart}

\usepackage[latin1]{inputenc}
\usepackage{amsmath}
\usepackage{amsfonts}
\usepackage{amssymb}
\usepackage{graphics}
\usepackage{enumerate}
\usepackage{amssymb,amsmath,amsthm,amscd,latexsym,verbatim,graphicx,amsfonts}

\usepackage{amscd}
\usepackage{amsmath}
\usepackage{amssymb}
\usepackage{amsthm}
\usepackage{latexsym}
\usepackage{verbatim}

\theoremstyle{plain}
\newtheorem{theorem}{Theorem}[section]

\newtheorem{lemma}[theorem]{Lemma}
\newtheorem{prop}[theorem]{Proposition}
\theoremstyle{definition}

\theoremstyle{remark}
\newtheorem{rem}[theorem]{Remark}

\newcommand{\nri}{n\rightarrow\infty}

\newcommand{\bbR}{\mathbb{R}}
\newcommand{\bbC}{\mathbb{C}}

\newcommand{\bbN}{\mathbb{N}}

\DeclareMathOperator*{\supp}{supp}
\DeclareMathOperator*{\ind}{ind}

\DeclareMathOperator*{\Imag}{Im}

\title[ ] {Analytic Versus Algebraic Density of Polynomials}

\author[]{Christian Berg, Brian Simanek and Richard Wellman}
\date{}

\begin{document}

\bibliographystyle{plain}

\thanks{  }

\maketitle

\begin{center}
\textit{Dedicated to the memory of Bent Fuglede (1925-2023)}
\end{center}

\begin{abstract}
We show that under very mild conditions on a measure $\mu$ on the interval $[0,\infty)$, the span of $\{x^k\}_{k=n}^{\infty}$ is dense in $L^2(\mu)$ for any $n=0,1,\ldots$.  We present two different proofs of this result, one based on the density index of Berg and Thill and one based on the Hilbert space $L^2(\mu)\oplus \bbC^{n+1}$.
Using the index of determinacy of Berg and Dur{\'a}n we prove that if the measure $\mu$ on $\bbR$ has infinite index of determinacy then the polynomial ideal $R(x)\bbC[x]$ is dense in $L^2(\mu)$ for any polynomial $R$ with zeros having no mass under $\mu$. 
\end{abstract}

\vspace{4mm}


\vspace{2mm}


\vspace{2mm}

\section{Introduction}\label{intro}

Polynomial approximation is a powerful tool with many applications in analysis.  Two of the most celebrated results bare the name of Weierstrass and concern density in the sup-norm.  The  Weierstrass Approximation Theorem states that on a compact interval $[a,b]\subseteq\bbR$, the polynomials are dense in the space $C_{\bbR}([a,b])$ in the sup-norm.  More generally, the Stone-Weierstrass Theorem states that if $X$ is a compact Hausdorff space and $S$ is a subalgebra of $C_{\bbR}(X)$ that separates points and includes the constant functions, then $S$ is dense in $C_{\bbR}(X)$ in the sup-norm (see \cite[Theorem 2.5.2]{SimI}).  A more general approximation result is the M{\"u}ntz Theorem, which states that the span of $\{1,x^{\lambda_1},x^{\lambda_2},\ldots\}$ is dense in $C[0,1]$ if and only if
\[
\sum_{j=1}^{\infty}\frac{\lambda_j}{\lambda_j^2+1}=\infty
\]
(see \cite[Theorem 4.2.1]{BE161}).  There are more general versions of this theorem that apply in various function spaces, all with their own hypotheses on the space or the accuracy of approximation (see \cite[Chapters 4 and 6]{BE161}).  

One can also consider polynomial approximation in settings where these theorems do not apply, such as approximation in $L^2$ spaces or using collections of polynomials that do not form a unital subalgebra of $C_{\bbR}(X)$.  This latter situation is especially relevant since the discovery of families of exceptional orthogonal polynomials in 2009 (see \cite{GKM09}).  These are families of polynomials that are analogous to classical orthogonal polynomials, but they do not include polynomials of every degree.  Nevertheless, they can still form an orthonormal basis for certain weighted $L^2$ spaces on the real line.  While this fact clearly does not create a contradiction with any known results, it is still surprising in that it highlights the significant difference between sets that are algebraically dense and sets that are analytically dense.

This is the starting point of the problem that we will investigate here.  Specifically, we will consider spaces of the form $L^2(\bbR,\mu)$ for some measure $\mu$ and density of sets of polynomials that omit polynomials of certain degrees.  The most basic such set is $\{x^k\}_{k=n}^{\infty}$ for some $n\in\bbN_0:=\{0,1,2,\ldots\}$.  This is the setting we will consider first.

Let $\mathcal M^*(\bbR)$ denote the set of  measures $\mu$ on the real line with moments of any order and infinite support. The first condition implies that that any polynomial $p\in\bbC[x]$ is square integrable with respect to $\mu$, and the second condition implies that two polynomials equal $\mu$-almost everywhere are identical. We can therefore consider $\bbC[x]$ as a subset of the Hilbert space $L^2(\mu)$. For $\mu\in\mathcal M^*(\bbR)$ the moments are given as
\begin{equation}\label{eq:moments}
s_n(\mu)=\int x^n\,d\mu(x),\quad n\in\bbN_0,
\end{equation}
and we say that $\mu$ is determinate if there is no other measure $\nu\in\mathcal M^*(\bbR)$ with the same moments as $\mu$.
A measure $\mu\in\mathcal M^*(\bbR)$ that is not determinate is called indeterminate, and in this case the set of measures having the same moments as $\mu$ is an infinite  convex set $V$ called the set of solutions to the moment problem \eqref{eq:moments}.  We recall the basic result of M. Riesz, cf. \cite{R23}, that $\bbC[x]$ is dense in $L^2(\mu)$ if $\mu$ is determinate, and if $\mu$ is indeterminate, then $\bbC[x]$ is dense in $L^2(\nu)$ for $\nu\in V$ if and only if $\nu$ is an N-extremal solution (also called a von Neumann solution in \cite{S}). These solutions are very special and always discrete and concentrated in the
zero set of an entire function of minimal exponential type , see \cite{Ak},\cite{S}. Further details about N-extremal solutions are given in Section 2.

Our first main result is the following:

\begin{theorem}\label{maindense2}
Let $\mu\in\mathcal M^*(\bbR)$ satisfy the following conditions:
\begin{itemize}
\item[(1)] $\mu$ is determinate,
\item[(2)] $\supp(\mu)\subseteq[0,\infty)$,
\item[(3)] $0\in\supp(\mu)$, 
\item[(4)] $\mu(\{0\})=0$.
\end{itemize}
Then for any $n\in\bbN_0$, the linear span  of the monomials $\{x^k\}_{k=n}^{\infty}$ is dense in $L^2(\mu)$.
\end{theorem}

\begin{rem}{\rm Note that (3) and (4) implies that $0$ is not an isolated point of $\supp(\mu)$.}
\end{rem}
We will present two proofs of Theorem~\ref{maindense2}.  The first builds on work presented in \cite{B:T2} and uses the notion of a density index inspired by the solution of the Challifour problem, see \cite{B:T1}.  For any measure $\mu\in\mathcal M^*(\bbR)$ supported on $[0,\infty)$ and for which the polynomials are dense in $L^2(\mu)$, we define its density index
\begin{equation}\label{eq:denseindex}
\delta(\mu)=\sup\{k\in\bbN_0:\bbC[x]\,\mbox{is dense in}\, L^2(x^kd\mu(x))\}
\end{equation}
as in \cite{B:T2}. Note that it is important that $\supp(\mu)\subseteq [0,\infty)$ for this definition because $x^kd\mu(x)$ shall be a positive measure.  Many classification results about measures with finite density index can be found there. Let us just recall that any determinate non-discrete $\mu\in\mathcal M^*(\bbR)$ supported by $[0,\infty)$ has $\delta(\mu)=\infty$. 

According to a theorem of M. Riesz, see \cite[p. 223]{R23}, we have that $\mu\in\mathcal M^*(\bbR)$ is determinate iff $\bbC[x]$ is dense in $L^2((1+x^2)d\mu(x))$, so for a determinate measure $\mu$ supported by $[0,\infty)$
we have $\delta(\mu)\ge 2$, because $1, 2x, x^2\le 1+x^2$ for $x\ge 0$.

Our proof of Theorem \ref{maindense2} that uses the density index is given in Section \ref{first}.

The second proof of our main result will use the Hilbert space $L^2(\mu)\oplus\bbC^{n+1}$ and can be found in Section \ref{second}.

In Section 4 we prove Theorem~\ref{thm:indinf} based on an index of determinacy for measures $\mu\in\mathcal M^*(\bbR)$. It states that if $\mu$ has infinite index of determinacy and $R(t)$ is any polynomial such that $\mu$ has no mass  at the zeros of $R$, then the polynomial ideal $R(x)\bbC[x]$ is dense in $L^2(\mu)$. Any non-discrete measure $\mu\in\mathcal M^*(\bbR)$ has infinite index of determinacy.

\section{A First Approach using the density index}\label{first}

Our first proof of Theorem \ref{maindense2} comes in two parts, the first of which is the following proposition.

\begin{prop}\label{sigmainf}
Let $\mu\in \mathcal M^*(\bbR)$ with $\supp(\mu)\subseteq[0,\infty)$ satisfy $\delta(\mu)=\infty$ and $\mu(\{0\})=0$.  Then for any $n\in\bbN_0$ the linear span of
 $\{x^k\}_{k=n}^\infty$ is dense in $L^2(\mu)$.
\end{prop}

\begin{proof}
 Let $n$ be a nonnegative integer and let $f\in L^2(\mu)$ and $\varepsilon>0$ be given. Since $\mu(\{0\})=0$ the function $f(x)x^{-n}$ is defined $\mu$-almost everywhere and belongs to  $L^2(x^{2n}d\mu(x))$, which has the property that the polynomials are dense. Therefore, there exists  a polynomium $p(x)$ such that
\begin{equation}
\varepsilon^2 >\int \left|f(x)x^{-n}-p(x)\right|^2 x^{2n}d\mu(x)=\int \left|f(x)-x^np(x)\right|^2\,d\mu(x),
\end{equation}  
which was to be proved.
\end{proof}

\medskip

Before we present the second part of the proof, let us review some facts about indeterminate measures.  If a measure $\mu\in\mathcal M^*(\bbR)$ is indeterminate, then the set $V$ of all measures on $\bbR$ that have the same moments as $\mu$
are given by the Nevanlinna parametrization, see \cite[p. 98]{Ak}, $V=\{\sigma_\varphi\}_{\varphi\in \mathcal N\cup\{\infty\}}$ defined by
\begin{equation}\label{eq:Nev}
\int\frac{d\sigma_\varphi(x)}{x-z}=-\frac{A(z)\varphi(z)-C(z)}{B(z)\varphi(z)-D(z)},\quad z\in\bbC\setminus\bbR,
\end{equation}
where the parameter set $\mathcal N$ is the set of Nevanlinna-Pick functions.  Such a function (also called a Herglotz function in the literature) is a holomorphic function $\varphi$ on the cut plane $\mathcal A:=\bbC\setminus\bbR$ satisfying $\Imag(\varphi(z))/\Imag(z)\ge 0$ for $z\in\mathcal A$. The functions $A,B,C,D$ are certain  entire functions defined in terms of the orthonormal polynomials $p_n$ and the polynomials $q_n$ of the second kind for $\mu$. (When $\varphi=\infty$ the right-hand side of \eqref{eq:Nev} is to be understood as $-A(z)/B(z)$.)

The N-extremal solutions are parametrized  by the Nevanlinna-Pick functions which are a real constant or infinity denoted $\sigma_t, t\in\bbR\cup\{\infty\}$.  
 
Several facts about these measures will be important for us:
\begin{itemize}

\item[(i)] Each $\sigma_t$ is discrete and supported on the zero set of the entire function $B(z)t-D(z)$, interpreted as $B(z)$ if $t=\infty$.
\item[(ii)] If the indeterminate moment problem is Stieltjes, i.e., there exists a solution $\sigma\in V$ supported by $[0,\infty)$, then there exists
a real number $\alpha\le 0$ such that $\supp(\sigma_t)\subseteq[0,\infty)$ if and only if $t\in[\alpha,0]$ (see \cite{B:V}). 
\item[(iii)] If $\alpha=0$ then there is one and only one solution in $V$ supported by $[0,\infty)$ namely the N-extremal measure $\sigma_0$. This moment problem is called determinate in the sense of Stieltjes, det(S) in short. 
\item[(iv)] If $\alpha<0$ then there are infinitely many solutions $\sigma_\varphi$ supported by $[0,\infty)$ and the corresponding functions $\varphi\in\mathcal N$ are characterized in \cite{P}. This moment problem is called indeterminate in the sense of Stieltjes, indet(S) in short.
\end{itemize}

Now we can present the second part of our first proof of Theorem \ref{maindense2}.

\begin{prop}\label{deltainf}
Assume that $\mu\in\mathcal M^*(\bbR)$ satisfies the hypotheses (1-4) in Theorem \ref{maindense2}.  Then $\delta(\mu)=\infty$.
\end{prop}

\begin{proof} Since $\mu$ is assumed determinate we have already noticed that $\delta(\mu)\ge 2$. Assume now that $2\le \delta(\mu)<\infty$. By \cite[Theorem 2.6]{B:T2} we know that
$$
\mu=c\delta_0+x^{2-\delta(\mu)}\,d\nu(x)
$$
for some $c\geq0$ and $\nu$ is a discrete measure supported on the strictly positive zeros of an entire function.  By (4) we have $c=0$, but then $0\not\in\supp(\mu)$, which is a contradiction.
\end{proof}

Combining Proposition~\ref{deltainf} with Proposition \ref{sigmainf} proves Theorem \ref{maindense2}.

\section{A Second Approach without use of the density index}\label{second}

 Let us begin by reviewing some facts about orthogonal polynomials and their zeros.

Suppose $\mu\in\mathcal M^*(\bbR)$ and let $\{p_n\}_{n=0}^{\infty}$ be the orthonormal polynomials for $\mu$, so that $\deg(p_n)=n$ and $p_n$ has positive leading coefficient.
\begin{itemize}
\item[(i)] The orthonormal polynomials for the measure $\mu$ have all of their zeros in $(a,b)$ for any interval $[a,b]\supseteq\supp(\mu)$.
\item [(ii)] The zeros of $p_n$ are simple.
\item[(iii)] The zeros of $p_n$ and $p_{n+1}$ strictly interlace.
\item[(iv)] If $\mu$ is determinate and $x_0\in\supp(\mu)$, then for every $\varepsilon>0$, the interval $(x_0-\varepsilon, x_0+\varepsilon)$ has a zero of $p_n$ for all sufficiently large $n$.
\end{itemize}
All of these facts can be found in \cite{Ch}.

Now we are ready to present a lemma that will be needed for this proof of Theorem \ref{maindense2}.

\begin{lemma}\label{gammarat} Assume that $\mu\in\mathcal M^*(\bbR)$ is determinate and
satisfies $\supp(\mu)\subseteq[0,\infty)$ and $0\in\supp(\mu)$ is not an isolated point of $\supp(\mu)$.  Let
\begin{equation}\label{eq:coef}
p_{n}(x)=\gamma_{n,n}x^n+\gamma_{n,n-1}x^{n-1}+\cdots+\gamma_{n,1}x+\gamma_{n,0}
\end{equation}
be the degree $n$ orthonormal polynomial for $\mu$ with positive leading coefficient.  For any fixed $k\in\bbN_0$ it holds that
\[
\lim_{\nri}\left|\frac{\gamma_{n,k}}{\gamma_{n,k+1}}\right|=0.
\]
\end{lemma}

\begin{proof}
To prove this, we shall consider the derivatives of $p_n$ at $0$.  Let
\[
x_{n,1}<x_{n,2}<\cdots<x_{n,n}
\]
denote the zeros of $p_n$.  Note that $0<x_{n,1}$. The expression
$$
p_n(x)=\gamma_{n,n}\prod_{j=1}^n (x-x_{n,j})
$$ 
shows that $(-1)^{n-j}\gamma_{n,j}>0, j=0,1,\ldots,n$, and in particular $\gamma_{n,j}\neq 0$.
Using the terminology of \cite[(5.6)]{Ch} we define $\xi_k=\lim_{n\to\infty}x_{n,k}, k=1,2,\ldots$. Clearly $\xi_1\le \xi_2\le \ldots$, and since  $0$ is a point of $\supp(\mu)$, it follows that $\xi_1=0$.

We have
\[
\frac{p_n'(0)}{p_n(0)}=\sum_{j=1}^n\frac{1}{-x_{n,j}}.
\]

and since $x_{n,1}\rightarrow0$ as $\nri$, we get
\[
\lim_{\nri}\left|\frac{p_n'(0)}{p_n(0)}\right|=\infty.
\]
In other words $|\gamma_{n,0}/\gamma_{n,1}|\rightarrow0$.

We now claim that $\xi_1=\xi_2=\cdots =0$. In fact, assuming $\xi_1=\cdots =\xi_k=0$ and $\xi_{k+1}>0$ we claim that $\mu((0,\xi_{k+1}))=0$, which contradicts that $0$ is not an isolated point in $\supp(\mu)$. To see the claim let $f$ be a  continuous function with compact support in $(0,\xi_{k+1})$ and satisfying $0\le f\le 1$. Let 
\begin{equation}\label{eq:Gauss} 
\pi_n=\sum_{j=1}^n H_{n,j}\delta_{x_{n,j}}
\end{equation}
denote the Gauss quadrature measure with masses $H_{n,j}>0$ at the zeros $x_{n,j}$ of $p_n$. By assumption $\mu$ is determinate and then we know that $\pi_n$ converges weakly to $\mu$ as $n\to\infty$. Therefore
$$
\int f(x)\,d\mu(x)=\lim_{n\to\infty} \sum_{j=1}^n H_{n,j}f(x_{n,j})=0
$$
because for $n$ sufficiently large we have $x_{n,j}<\inf\supp(f)$ for $j=1,\ldots,k$ and
$x_{n,j}\ge \xi_{k+1}$ for $j=k+1,\ldots,n$. Since $f$ is arbitrary  with the properties stated, we get $\mu((0,\xi_{k+1}))=0$ as claimed. 

We now look at the zeros of $p_n'$. Clearly $p_n'$ has  a zero $x'_{n,j} \in (x_{n,j}, x_{n,j+1}), j=1,\ldots,n-1$, so $p_n'$ has n-1 simple zeros $x'_{n,1}<x'_{n,2}<\cdots<x'_{n,n-1}$. In the same way we see that $p^{(j)}_n$ has $n-j$ strictly positive zeros
$$
x^{(j)}_{n,1}<x^{(j)}_{n,2}<\cdots<x^{(j)}_{n,n-j},\quad j<n,
$$
and
$$
\frac{p^{(j+1)}_n(0)}{p^{(j)}_n(0)}=\sum_{k=1}^{n-j}\frac{1}{-x^{(j)}_{n,k}}.
$$
Using that $x_{n,j+1}\rightarrow0$ as $\nri$ for every  fixed $j$, we see that $x^{(j)}_{n,1}\rightarrow0$ as $\nri$.  Therefore
\[
\lim_{\nri}\left|\frac{p^{(j+1)}_n(0)}{p^{(j)}_n(0)}\right|=\infty.
\]
In other words $|\gamma_{n,j}/\gamma_{n,j+1}|\rightarrow0$ for $n\to\infty$.
\end{proof}

\begin{rem}\label{thm:SW}  {\rm It is important for the conclusion of Lemma~\ref{gammarat} that $\mu$ is determinate. Let us consider the Stieltjes-Wigert polynomials associated with the lognormal distribution. We follow the notation of \cite{Chris}. For $0<q<1$ consider the probability measure $\mu$ with density on the interval $(0,\infty)$ given as
$$
w(x)=\frac{q^{1/8}}{2\pi\log(q^{-1})}\frac{1}{\sqrt{x}}\exp\left(\frac{(\log x)^2}{2\log q}\right),\quad x>0.
$$  
It is indeterminate with $\supp(\mu)=[0,\infty)$. The orthonormal polynomials are given as
$$
p_n(x)=(-1)^n\sqrt{\frac{q^n}{(q;q)_n}}\sum_{k=0}
^n \frac{(q;q)_n}{(q;q)_k (q;q)_{n-k}}(-1)^k q^{k^2}x^k,
$$
where we use the notation 
$$
(q;q)_n=\prod_{k=1}^n  (1-q^k).
$$
Therefore
$$
\frac{\gamma_{n,k}}{\gamma_{n,k+1}}=-\frac{1-q^{k+1}}{(1-q^{n-k})q^{2k+1}},
$$
which has a non-zero limit when $n\to\infty$.}
\end{rem}

Before we finish our proof, we need to mention one additional fact.  For a determinate measure $\mu$, it is true that
\begin{equation}\label{eq:mass}
\left(\sum_{n=0}^{\infty}|p_n(x)|^2\right)^{-1}=\mu(\{x\}),
\end{equation}
where we interpret $(\infty)^{-1}=0$. This is Parseval's formula for the indicator function $1_{x}$.

\begin{proof}[Second Proof of Theorem \ref{maindense2}]
We  fix $n\in\bbN_0$ and define the Hilbert space $H=L^2(\mu)\oplus\bbC^{n+1}$ with inner product
\[
\langle(f,\beta),(g,\gamma)\rangle=\int f(x)\overline{g(x)}\,d\mu(x)+\gamma^*\beta
\]
To each polynomial $p$, we can associate the vector in $H$ given by
\begin{equation}\label{pembed}
\vec{p}=(p,(p(0),p^{(1)}(0),\ldots,p^{(n)}(0))).
\end{equation}
Our first task is to prove that the collection of all such vectors (for all polynomials $p$) is dense in $H$.

To prove this claim, suppose $(f,\beta)\in H$ is orthogonal to all vectors of the form \eqref{pembed}.  As in \eqref{eq:coef} we write
\begin{equation}\label{pform}
p_r(x)=\sum_{j=0}^r\gamma_{r,j}x^j
\end{equation}
for the degree $r$ orthonormal polynomial for the measure $\mu$ with positive leading coefficient.  Then we find
\[
\int p_r(x)f(x)d\mu(x)=-\sum_{k=0}^{n}k!\beta_k\gamma_{r,k}
\]
for all $r\in\bbN_0$ with the convention that $\gamma_{r,k}=0$ when $k>r$. We start by proving $\beta_n=0$. Assume $\beta_n\neq 0$ for obtaining a contradiction.  
We have for $r\ge n$
$$
\left|\sum_{k=0}^n k! \beta_k\gamma_{r,k}\right| \ge |\gamma_{r,n}|\left(n!|\beta_n|
-\left|\sum_{k=0}^{n-1} k!\beta_k\gamma_{r,k}/\gamma_{r,n}\right|\right),
$$
and by  Lemma \ref{gammarat} the last term tends to $0$ for $r\to\infty$. For $r$ sufficiently large we therefore have
\[
\left|\sum_{k=0}^{n}k!\beta_k\gamma_{r,k}\right|>\frac{n!}{2}|\beta_n\gamma_{r,n}|>\frac{n!}{2}|\beta_n\gamma_{r,0}|.
\]
Therefore,
\[
\frac{n!}{2}|\beta_n\gamma_{r,0}|\leq\left|\int p_r(x)f(x)d\mu(x)\right|
\]
for $r$ sufficiently large.
However, Bessel's inequality asserts that the terms on the right-hand side of this inequality form a square summable sequence.  The terms on the left-hand side are not square summable when $\beta_n\neq0$ because $\mu(\{0\})=0$, cf. \eqref{eq:mass}. This gives a contradiction so $\beta_n=0$.  We can then proceed step by step as above to derive that each $\beta_j=0$ for all $j=0,\ldots,n$.  We conclude that $f$ is orthogonal to each $p_r$ in $L^2(\mu)$.  Since we are know that the polynomials are dense in $L^2(\mu)$, this tells us that $f=0$ as desired.

Now, take any $g\in L^2(\mu)$.  By our previous calculation, we know that we can find polynomials $\{a_k\}_{k\in\bbN}$ so that $\vec{a}_k\rightarrow (g,\vec{0})$ as $k\rightarrow\infty$, where $\vec{0}$ is the zero vector in $\bbC^{n+1}$.  Writing
\[
a_k(x)=b_k(x)+x^{n+1}r_k(x)
\]
with $\deg(b_k)\leq n$, we have
\[
(b_k(0),b_k^{(1)}(0),\ldots,b_k^{(n)}(0))=(a_k(0),a_k^{(1)}(0),\ldots,a_k^{(n)}(0))\rightarrow\vec{0}
\]
as $k\rightarrow\infty$.  This implies $b_k\rightarrow0$ in $L^2(\mu)$ as $k\rightarrow\infty$.  We also have $\|g- a_k\|_{L^2(\mu)}\rightarrow0$ as $k\rightarrow\infty$ and hence the triangle inequality implies $\|g- x^{n+1}r_k\|_{L^2(\mu)}\rightarrow0$ as $k\rightarrow\infty$, which is what we wanted to show.
\end{proof}

\section{Measures of infinite index of determinacy}\label{det}
 For a determinate measure
$\mu\in\mathcal M^*(\bbR)$ and a complex number $z\in\bbC$ Berg and Dur{\'a}n  introduced in \cite{B:D1}
\begin{equation}
{\ind}_z(\mu):=\sup\{k\in\bbN_0: |x-z|^{2k}\,d\mu(x) \;\mbox{is determinate}\}
\end{equation}  
called the index of determinacy of $\mu$ at the point $z$. They proved that
${\ind}_z(\mu)$ is constant $k\in\bbN_0\cup\{\infty\}$ for $z$ in the $\bbC\setminus\supp(\mu)$ and constant equal to $k+1$ for $z\in\supp(\mu)$. If $\mu$ is a non-discrete determinate measure, it turns out that ${\ind}_z(\mu)=\infty$ for all $z\in\bbC$. 

Finally in \cite{B:D2} the authors decided to use the notation 
\begin{equation}\label{eq:ind}
{\ind}(\mu):={\ind}_z(\mu),\quad z\in\bbC\setminus\supp(\mu),
\end{equation}
for the index of determinacy of $\mu$.

A main result Theorem 3.6 in \cite{B:D1} explains how to obtain a measure with finite index of determinacy. One starts with an N-extremal indeterminate measure
$\tau\in\mathcal M^*(\bbR)$. Such a measure is discrete supported in countably many points of the real axis, and these points are the zero set of a certain entire function of minimal exponential type as  explained in Section \ref{first}.  Let $k\in\bbN_0$. If the measure $\mu$ is obtained  from $\tau$ by removing $k+1$ mass points of $\tau$ then $\mu$ is determinate with  ${\ind}(\mu)=k$ and
\begin{equation}
{\ind}_z(\mu)=\left\{\begin{array}{cl} k, &  z\in\bbC\setminus \supp(\mu)\\
k+1,& z\in\supp(\mu).
\end{array}
\right.    
\end{equation}

Furthermore, Theorem 3.9 in \cite{B:D1} states that any determinate measure $\mu$ with ${\ind}(\mu)=k\in\bbN_0$ can be obtained in this way.

For a system $(z_1,k_1),\ldots, (z_N,k_N)$, where the $z$'s are $N$ different complex numbers  and the $k$'s are nonnegative integers, we introduce the polynomial
\begin{equation}\label{eq:system}
R(x):=\prod_{j=1}^N (x-z_j)^{k_j+1}  \quad\mbox{of degree}\quad M=\sum_{j=1}^N (k_j+1).
\end{equation}

\begin{theorem}\label{thm:indinf} Let $\mu\in\mathcal M^*(\bbR)$ be a determinate measure with  index ${\ind}(\mu)=\infty$, and consider the system $(z_j,k_j), j=1,\ldots,N$  above with polynomial $R$ of degree $M$ as in \eqref{eq:system}.

If $\mu(\{z_1,\ldots,z_N\})=0$ and in particular if $z_j\in\bbC\setminus\bbR, j=1,\ldots,N$ then
$$
\overline{R(x) \bbC[x]}=L^2(\mu).
$$ 
\end{theorem} 

\begin{proof} Let $f\in L^2(\mu)$ and $\varepsilon>0$ be given. Then $f(x)/R(x)$ is defined $\mu$ almost everywhere and it belongs to $L^2(|R(x)|^2\,d\mu(x))$ in which the polynomials are dense because the measure $|R(x)|^2\,d\mu(x)$ is determinate. Therefore there exists a polynomial $p(x)$ such that
$$
\int |f(x)/R(x)-p(x)|^2 |R(x)|^2\,d\mu(x)<\varepsilon^2,
$$ 
i.e.,
$$
\int |f(x)-R(x)p(x)|^2 d\mu(x)<\varepsilon^2.
$$
\end{proof}

This is a version of Proposition 2.9 in \cite{B:D2} adapted to our purpose. Notice also that the famous Theorem of Marcel Riesz (\cite[p. 223]{R23}) can be stated:

``If $\mu$ is determinate then the polynomials of the form $(x-z)p(x), p\in\bbC[x]$ are dense in $L^2(\mu)$ for an arbitrary $z\in\bbC\setminus\bbR$."

\begin{rem}\label{thm:detdelta}  {\rm Theorem~\ref{maindense2} can be deduced from Theorem~\ref{thm:indinf}. In fact, if the hypotheses (1-4) in Theorem~\ref{maindense2} are satisfied, then we know from Proposition~\ref{deltainf}
that $\delta(\mu)=\infty$. This implies that ${\ind}_0(\mu)=\ind(\mu)=\infty$ for if
${\ind}_0(\mu)<\infty$ then $\tau:=x^{{\ind}_0(\mu)+2}\,d\mu(x)$ is indeterminate and hence N-extremal since $\delta(\mu)=\infty$. Clearly $\tau(\{0\})=0$ so by (iii) in Section 2 we see that $\tau$ must be indet(S). However, by \cite[Theorem 2.1]{B:T2} we get $\delta(\tau)\in\{0,1\}$, which is a contradiction to $\delta(\tau)=\delta(\mu)=\infty$.}
\end{rem}

\bigskip

\noindent\textbf{Acknowledgements.}  It is a pleasure to thank Lance Littlejohn for encouraging us to pursue this line of research.

\bigskip

\bigskip

\noindent Christian Berg\\
Department of Mathematical Sciences, University of Copenhagen\\
Universitetsparken 5, DK-2100 Copenhagen, Denmark\\
email: {\tt{berg@math.ku.dk}}

\bigskip

\noindent Brian Simanek\\
Department of Mathematics, Baylor University\\
Waco, TX, USA\\
email: {\tt{Brian$\_$Simanek@baylor.edu}}

\bigskip

\noindent Richard Wellman\\

\end{document}